\documentclass[12pt]{amsproc}
\usepackage{amsaddr, biblatex, hyperref, tikz, doi}
\hypersetup{
           breaklinks=true,   % splits links across lines
           colorlinks=true,   % displays links as colored text instead of blocks
           pdfusetitle=true,  % \title and \author values into pdf metadata etc.
           citecolor=blue
        }

\DeclareFieldFormat{doi}{\href{https://dx.doi.org/#1}{\textsc{doi}}}
\DeclareFieldFormat{url}{\href{#1}{\textsc{url}}}
\renewbibmacro*{doi+eprint+url}{
  \printfield{doi}
  \newunit\newblock
  \iftoggle{bbx:eprint}{
    \usebibmacro{eprint}
  }{}
  \newunit\newblock
  \iffieldundef{doi}{
    \usebibmacro{url+urldate}}
  {}
}
\renewbibmacro{in:}{}
\AtEveryBibitem{
  \ifentrytype{book}
  {\clearfield{pages}}
}

\makeatletter
\@namedef{subjclassname@2020}{\textup{2020} Mathematics Subject Classification}
\makeatother

\title[Splitting subspaces and the Touchard-Riordan formula]{Splitting subspaces and a finite field interpretation of the Touchard-Riordan Formula}
\author{Amritanshu Prasad}
\address{The Institute of Mathematical Sciences, Chennai, India.}
\address{Homi Bhabha National Institute, Mumbai, India.}
\email{amri@imsc.res.in}

\author{Samrith Ram}
\address{Indraprastha Institute of Information Technology Delhi, New Delhi, India.}
\email{samrith@gmail.com}
\subjclass[2020]{05A15,05A19,33C45}
\keywords{Touchard-Riordan formula, splitting subspaces, finite fields, $q$-Hermite orthogonal polynomials, chord diagrams}

\newtheorem*{theorem*}{Main Theorem}
\newtheorem{lemma}{Lemma}
\newtheorem{corollary}[lemma]{Corollary}
\newtheorem{theorem}[lemma]{Theorem}
\newtheorem{proposition}[lemma]{Proposition}
\theoremstyle{remark}
\newtheorem{example}[lemma]{Example}
\newtheorem{remark}[lemma]{Remark}
\theoremstyle{definition}
\newtheorem{definition}[lemma]{Definition}
\newtheorem{claim}{Claim}

\numberwithin{equation}{section}
\numberwithin{lemma}{section}

\DeclareMathOperator\Irr{Irr}
\DeclareMathOperator\Par{Par}

\newcommand\Fq{\mathbf F_q}

\newcommand\qbin[2]{{#1\brack#2}_q}
\renewcommand\aa{\mathbf a}
\newcommand\bb{\mathbf b}
\newcommand\jj{\mathbf j}
\newcommand\kk{\mathbf k}

\newcommand\ZZ{\mathbf Z}
\usepackage{hyperref} 
\begin{document}
\begin{abstract}
  We enumerate the number of $T$-splitting subspaces of dimension $m$ for an arbitrary operator $T$ on a $2m$-dimensional vector space over a finite field.
  When $T$ is regular split semisimple, comparison with an alternate method of enumeration leads to a new proof of the Touchard-Riordan formula for enumerating chord diagrams by their number of crossings.
\end{abstract}
\maketitle
\section{Introduction}
Let $\Fq$ denote a finite field of order $q$, and $m$ be a non-negative integer.
Given a positive integer $d$ and a linear operator $T:\Fq^{dm}\to \Fq^{dm}$, an $m$-dimensional subspace $W\subset \Fq^{dm}$ is said to be \emph{$T$-splitting} if
\begin{displaymath}
  W + TW + \dotsb + T^{d-1}W = \Fq^{dm}.
\end{displaymath}
This definition was proposed by Ghorpade and Ram \cite{m2}, motivated by the work of Niederreiter \cite{N2}.

The number of $T$-splitting subspaces is known when $T$ has an irreducible characteristic polynomial \cite{MR4263652,sscffa,m2}, is regular nilpotent \cite{agram2022}, is regular split semisimple \cite{fpsac,pr}, or when the invariant factors satisfy certain degree constraints \cite{polynomialmatrices}.
In this article, we consider the case where $d=2$.
Our main theorem gives a formula for the number of $T$-splitting subspaces of dimension $m$ for any $T\in M_{2m}(\Fq)$.
\begin{theorem*}
  For any linear operator $T:\Fq^{2m}\to \Fq^{2m}$, the number of $m$-dimensional $T$-splitting subspaces of $\Fq^{2m}$ is given by
  \begin{displaymath}
    \sigma^T = q^{\binom m2}\sum_{j=0}^{2m} (-1)^j X_j^T q^{\binom{m-j+1}2},
  \end{displaymath}
  where $X_j^T$ is the number of $j$-dimensional $T$-invariant subspaces of $\Fq^{2m}$.
\end{theorem*}
The quantities $X_j^T$ are easy to compute from the Jordan canonical form of $T$ with the help of a recursive formula of Ramar\'e \cite{MR3611779}.
For a detailed discussion see Section~\ref{sec:computation-x_j}.
When $T$ is regular split semisimple (i.e., it is similar to a diagonal matrix with distinct diagonal entries), $X_j^T=\binom{2m}j$, so the number of $T$-splitting subspaces is
\begin{displaymath}
  \sigma^T = q^{\binom m2}\sum_{j=0}^{2m} (-1)^j\binom{2m}jq^{\binom{m-j+1}2}.
\end{displaymath}
The sum above is the right hand side of the Touchard-Riordan formula
\begin{equation}
  \label{eq:touchard-riordan}
  (q-1)^m T_m(q) = \sum_{j=0}^{2m} (-1)^j\binom{2m}jq^{\binom{m-j+1}2}
\end{equation}
for the polynomial $T_m(q)$ that enumerates chord diagrams on $2m$ nodes according to their number of crossings (see Section~\ref{sec:chord-diags}).
This identity is attributed to Touchard \cite{MR46325} and Riordan \cite{MR366686}.
A proof using the theory of continued fractions was given by Read~\cite{MR556055}, and a bijective proof was given by Penaud~\cite{MR1336847}.
The polynomials $T_m(q)$ are moments of the $q$-Hermite orthogonal polynomial sequence \cite[Prop.~4.1]{MR930175}.
Several generalizations and variations of the Touchard-Riordan formula can be found in \cite{MR2737181,MR2721522,MR2799608,MR3033681,MR2819649}.

When $T\in M_{2m}(\Fq)$ is regular split semisimple, splitting subspaces can also be enumerated (see Theorem~\ref{theorem:split-splitting}) using the technique of \cite[Section~4.6]{pr} as 
\begin{displaymath}
  \sigma^T = q^{\binom m2}(q-1)^mT_m(q).
\end{displaymath}
This gives a completely new self-contained proof of the Touchard-Riordan formula.
Thus our main theorem could be viewed as a generalisation of the Touchard-Riordan formula in the setting of finite fields.

A software demonstration of our results using SageMath \cite{sagemath} can be found at \url{https://www.imsc.res.in/~amri/splitting_subspaces/}.
\section{Enumeration of Invariant Subspaces}
\label{sec:computation-x_j}
Each $T\in M_n(\Fq)$ gives rise to an $\Fq[t]$-module $M_T$ with underlying vector space $\Fq^n$ on which $t$ acts by $T$.
A subspace of $\Fq^n$ is $T$-invariant if and only if it is a submodule of $M_T$.
Let $\Par$ denote the set of all integer partitions and $\Irr \Fq[t]$ denote the set of all irreducible monic polynomials in $\Fq[t]$.
By the theory of elementary divisors (see \cite[Section~3.9]{MR780184} and \cite[Section~1]{MR72878}) there exists a unique function $c_T:\Irr\Fq[t]\to \Par$ such that
\begin{equation}
  \label{eq:primary}
  M_T = \bigoplus_{p\in \Irr\Fq[t]} M_{T_p},
\end{equation}
with the $p$-primary component $M_{T_p}$ having structure
\begin{equation}
  \label{eq:local_jcf}
  M_{T_p} = \bigoplus_i\Fq[t]/(p(t)^{c_T(p)_i})
\end{equation}
where $c_T(p)_1,c_T(p)_2,\dotsc$ are the parts of the partition $c_T(p)$.

Define the \emph{invariant subspace generating function} $f_T$ of $T\in M_n(\Fq)$ as
\begin{displaymath}
  f_T(t) = \sum_{j=0}^n X_j^T t^j,
\end{displaymath}
where $X_j^T$ is the number of $j$-dimensional $T$-invariant subspaces of $\Fq^n$.
Each $\Fq[t]$-submodule of $M_T$ is uniquely expressible as a direct sum of submodules of the primary submodules $M_{T_p}$.
Therefore,
\begin{displaymath}
  f_T(t) = \prod_{p\in \Irr\Fq[t]} f_{T_p}(t).
\end{displaymath}

For each $\lambda\in \Par$, let $f_\lambda(q;t)$ denote the invariant subspace generating function of the nilpotent matrix over $\Fq$ whose Jordan block sizes are the parts of $\lambda$.
A surprisingly simple recurrence of Ramar\'e~\cite[Theorem~3.1]{MR3611779} allows for easy computation of $f_\lambda(q;t)$:
\begin{equation}
  \label{eq:ramare}
  (t-1)f_\lambda(q;t) = t^{\lambda_1+1}q^{\sum_{j\geq 2}\lambda_j}f_{(\lambda_2,\lambda_3,\dotsc)}(q;t/q) - f_{(\lambda_2,\lambda_3,\dotsc)}(q;tq),
\end{equation}
where $\lambda_1,\lambda_2,\dotsc$ are the parts of $\lambda$ in weakly decreasing order.
The empty partition $\emptyset$ of $0$ can be used as the base case with $f_\emptyset(q;t)=1$.
The recurrence \eqref{eq:ramare} implies that $f_\lambda(q;t)$ is a polynomial in $q$ and $t$ with integer coefficients.

Since the rings $\Fq[t]/p(t)^d$ and $\mathbf F_{q^d}[u]/u^d$ are isomorphic, the invariant subspace generating function of $T\in M_n(\Fq)$ is given by
\begin{equation}
  \label{eq:fT}
  f_T(t) = \prod_{p\in \Irr\Fq[t]} f_{c_T(p)}(q^{\deg p};t^{\deg p}).
\end{equation}

It follows that the polynomial $f_T(t)$ depends on the polynomials $p\in \Irr\Fq[t]$ only through their degrees.
\begin{definition}
  \label{definition:similarity-class-type}
  A \emph{similarity class type} of size $n$ is a multiset $\tau$ of pairs of the form $(d,\lambda)$ where $d$ is a positive integer and $\lambda$ is a non-empty integer partition such that $\sum_{(d,\lambda)\in\tau} d|\lambda|=n$ (the sum is taken with multiplicity).
  The similarity class type of $T\in M_n(\Fq)$ is the similarity class of size $n$ given by
  \begin{displaymath}
    \{(\deg(p),c_T(p))\mid p\in \Irr\Fq[t],\;c_T(p)\neq \emptyset\}.
  \end{displaymath}
\end{definition}
\begin{remark}
  The set of similarity class types of size $n$ is independent of $q$.
  Green~\cite{MR72878} introduced similarity class types to organise conjugacy classes of $GL_n(\Fq)$ in a manner independent of $q$.
  This enabled him give a combinatorial description of the character table of $GL_n(\Fq)$ across all $q$.
  For a detailed discussion and a software implementation see \cite{simsage}.
\end{remark}
\begin{example}
  \begin{enumerate}
  \item An $n\times n$ scalar matrix has similarity class type $\{1,(1^n)\}$.
  \item
    A regular split semisimple $n\times n$ matrix has similarity class type\linebreak $\{(1,(1)),\dotsc,(1,(1))\}$ (with $n$ repetitions).
  \item A regular nilpotent $n\times n$ matrix has type $\{(1,(n))\}$.
  \item  An $n\times n$ matrix with irreducible characteristic polynomial has type $\{(n,(1))\}$.
  \end{enumerate}
\end{example}
\begin{theorem}
  \label{theorem:type-dependence}
  Given a similarity class type $\tau$ of size $n$ and $0\leq j\leq n$ let
    \begin{displaymath}
    f_\tau(u;t) = \prod_{(d,\lambda)\in \tau} f_\lambda(u^d;t^d).
  \end{displaymath}
  Then for any prime power $q$ and any matrix $T\in M_n(\Fq)$ with similarity class type $\tau$, $f_T(t)=f_\tau(q;t)$.
  In particular, for every $0\leq j\leq n$, there exists a polynomial $X^\tau_j(u)\in \ZZ[u]$ such that $X^T_j = X^\tau_j(q)$.
\end{theorem}
\begin{proof}
  The theorem follows from Eqns.~(\ref{eq:ramare}) and (\ref{eq:fT}).
\end{proof}
The polynomial $f_\lambda(q;t)$ is known to have non-negative coefficients \cite{MR1223236}, hence $X_j^\tau(q)$ also has non-negative coefficients.
\begin{example}
  \label{example:taui}
  Let $\tau_i=\{(1,(1^{m+i})),(m-i,(1))\}$ for $i=1,\dotsc,m$.
  Then
  \begin{displaymath}
    f_{\tau_i}(t) = \left(\sum_{k=0}^{m+i}\qbin{m+i}k t^k\right)(1+t^{m-i}).
  \end{displaymath}
  Consequently,
  \begin{equation}
    \label{eq:taui}
    X^{\tau_i}_j(q) = \qbin{m+i}j + \qbin{m+i}{j-m+i}.
  \end{equation}
\end{example}
\section{The Existence of a Formula}
\label{sec:existence}
In this section we establish the existence of a formula for the number $\sigma^T$ of $m$-dimensional $T$-splitting subspaces of $\Fq^{2m}$ in terms of $X_j^T$, $j=0,\dotsc,m$ (Corollary~\ref{cor:ajxj}).
The main step is Proposition~\ref{prop:recurrence}, which is a special case of a more general recurrence of Chen and Tseng \cite[Lemma~2.7]{sscffa}.

Given a positive integer $n$, and $0\leq a\leq n$, let $\aa$ denote the set of $a$-dimensional subspaces of $\Fq^n$.
Given a linear operator $T:\Fq^n\to \Fq^n$ and sets $X$ and $Y$ of subspaces of $\Fq^n$, define
\begin{align*}
  (X,Y)_T&:=\{W \in X\mid W\cap T^{-1}W\in Y\}\\
  [X,Y]_T&:=\{(W_1,W_2)\mid W_1\in X, W_2\in Y,\text{ and } W_1\cap T^{-1}W_1\supset W_2\}.
\end{align*}
Thus $(\aa,\bb)_T$ denotes the set of $a$-dimensional subspaces $W$ such that $W\cap T^{-1}W$ has dimension $b$.
We drop the subscript $T$ from the notation when the operator is clear from the context.
\begin{example}
  For each $0\leq a\leq n$, $(\aa,\aa)$ denotes the set of $a$-dimensional $T$-invariant subspaces of $\Fq^n$.
  Hence $|(\aa,\aa)_T|=X_a^T$.
\end{example}
\begin{example}
  If $n=2m$, then $(\mathbf m,\mathbf 0)_T$ is the set of $m$-dimensional $T$\nobreakdash-splitting subspaces of $\Fq^{2m}$.
\end{example}
\begin{proposition}
  \label{prop:recurrence}
  Let $T:\Fq^n\to \Fq^n$ be a linear map.
  For all integers $n\geq a>b\geq 0$, we have
  \begin{align*}
    |(\aa,\bb)| & = X_b^T\qbin{n-b}{a-b}-X_a^T\qbin ab\\
                & + \sum_{j=0}^{b-1}|(\bb,\jj)|{n-2b+j \brack a-2b+j}_q-\sum_{k=b+1}^{a-1}|(\aa,\kk)|{k \brack b}_q.
  \end{align*}
\end{proposition}
\begin{proof}
  Since $\aa = \coprod_{0\leq a\leq k} (\aa,\kk)$, we have
  \begin{align*}
    [\aa,\bb]=\coprod_{b\leq k\leq a}[(\aa,\kk),\bb].
  \end{align*}
  It follows that
  \begin{align}
    |[\aa,\bb]|&=\sum_{k=b}^a |[(\aa,\kk),\bb]|\nonumber \\
           &=\sum_{k=b}^a|(\aa,\kk)|{k \brack b}_q\nonumber\\
     &=|(\aa,\bb)|+\sum_{k=b+1}^a|(\aa,\kk)|{k \brack b}_q. \label{eq:1}
  \end{align}
  Similarly,
  \begin{align*}
    [\aa,\bb]=\coprod_{0\leq j\leq b}[\aa,(\bb,\jj)],
  \end{align*}
  so that
  \begin{align}
    |[\aa,\bb]|&=\sum_{j=0}^b |[\aa,(\bb,\jj)]|\nonumber \\
    &=\sum_{j=0}^b|(\bb,\jj)|{n-(2b-j) \brack a-(2b-j)}_q. \label{eq:2}
  \end{align}
The proposition follows from Eqs.~\eqref{eq:1} and \eqref{eq:2}, and $|(\aa,\aa)|=X_a^T$.
\end{proof}
\begin{proposition} 
  For all integers integer $n\geq a\geq b\geq 0$, there exist polynomials $p_0(t),\dotsc,p_a(t)\in \ZZ[t]$ such that, for every prime power $q$, and every linear map $T:\Fq^n\to \Fq^n$,
  \begin{align*}
    |(\aa,\bb)_T|=\sum_{j=0}^a p_j(q)X_j^T.
  \end{align*}
\end{proposition}
\begin{proof}
  Proposition~\ref{prop:recurrence} expands $|(\aa,\bb)_T|$ in terms of $X_a^T$, $X_b^T$, and $|(\aa',\bb')_T|$ where either $a'<a$, or $a'=a$ and $a'-b'<a-b$.
  The coefficients are polynomials in $q$ that are independent of $T$.
  Thus repeated application of Proposition~\ref{prop:recurrence} will result in an expression of the stated form in finitely many steps.
\end{proof}
\begin{corollary}
  \label{cor:ajxj}
  For each non-negative integer $m$, there exist\linebreak polynomials $p_0(t),\dotsc,p_m(t)\in \ZZ[t]$ such that, for every linear map $T:\Fq^{2m}\to \Fq^{2m}$, the number of $m$-dimensional $T$-splitting subspaces is given by
  \begin{equation}
    \label{eq:ajxj} 
    \sigma^T = \sum_{j=0}^m p_j(q)X_j^T.
  \end{equation}
\end{corollary}

\section{Proof of the Main Theorem}
\label{sec:proof-main-theorem}
By Theorem~\ref{theorem:type-dependence} and Corollary~\ref{cor:ajxj}, for every similarity class type $\tau$ of size $2m$, there exists $\sigma^\tau(u)\in \ZZ[u]$ such that, for every prime power $q$ and every $T\in M_{2m}(\Fq)$ of type $\tau$, $\sigma^T = \sigma^\tau(q)$.
Thus the main theorem can be rephrased as follows.
 \begin{theorem}
   \label{th:main} 
  For each similarity class type $\tau$ of size $2m$,
  \begin{equation}
    \label{eq:sigma-tau}
    \sigma^{\tau}(q)=q^{m \choose 2}   \sum_{j=0}^{2m} (-1)^jX_j^\tau(q)q^{\binom{m-j+1}2}.
  \end{equation}
\end{theorem}
\subsection*{The proof strategy}
Since the lattice of submodules of $M^T$ is self-dual, $X_j^\tau(q)=X_{2m-j}^\tau(q)$. Therefore the right hand side of (\ref{eq:sigma-tau}) can be rewritten in terms of $X^\tau_0(q),\dotsc,X^\tau_m(q)$, bringing it to the form \eqref{eq:ajxj}.

Suppose $\tau_0,\dotsc,\tau_m$ are similarity class types of size $2m$ such that the determinant $(X^{\tau_i}_j(q))_{0\leq i,j\leq m}$ is non-zero.
Then the system of equations
\begin{displaymath}
  \sigma^{\tau_i}(q) = \sum_{j=0}^m p_j(q)X^{\tau_i}_j(q), \quad i=0,\dotsc,m
\end{displaymath}
has a unique solution for the $p_j(q)$.
Thus, if we prove \eqref{eq:sigma-tau} for $\tau=\tau_0,\dotsc,\tau_m$, we will have shown that Theorem~\ref{th:main} holds in general.

Take $\tau_0=\{(2m,(1))\}$, the type of a simple matrix (a matrix with irreducible characteristic polynomial), and for $i=1,\dotsc,m$, take $\tau_i=\{(1,(1^{m+i})),(m-i,(1))\}$.
The proof of Theorem~\ref{th:main} is reduced to the following steps:

\begin{claim}
  \label{claim:1}
  The formula \eqref{eq:sigma-tau} holds for $\tau=\tau_0,\dotsc,\tau_m$.
\end{claim}
\begin{claim}
  \label{claim:2}
  The determinant of $X=(X^{\tau_i}_j(q))_{0\leq i,j\leq m}$ is non-zero.
\end{claim}
\subsection*{Proof of Claim~\ref{claim:1}}
Consider first $\tau=\tau_0$.
It is shown in \cite[Theorem 1.4]{MR4263652} that
\begin{displaymath}
  \sigma_{\tau_0}(q) = q^{\binom m2}(q^{\binom{m+1}2}+q^{\binom m2}).
\end{displaymath}
On the  other hand, $\tau_0$ is the type of a simple matrix, so $X^\tau_j(q) = 1$ if $j=0$ or $2m$, and $X^\tau_j(q)=0$ for $0<j<2m$.
Therefore
\begin{displaymath}
  \sum_{j=0}^{2m} (-1)^jX_j^\tau(q)q^{\binom{m-j+1}2} = q^{\binom m2}(q^{\binom{m+1}2} + q^{\binom{-m+1}2}) = q^{\binom m2}(q^{\binom{m+1}2} + q^{\binom m2}),
\end{displaymath}
establishing \eqref{eq:sigma-tau} for $\tau_0$.

For $i=1,\dotsc,m$, $\sigma^{\tau_i}(q) = 0$ since any $T\in M_n(\Fq)$  of type $\tau_i$ satisfies the hypothesis of the following general lemma.
\begin{lemma}
  Let $l(\lambda)$ denote the number of parts of an integer partition $\lambda$.
  If $W\subset \Fq^n$ is such that $\sum_{j\geq 0} T^jW=\Fq^n$, then $\dim W\geq l(c_T(p))$ for all $p\in \Irr\Fq[t]$.
  In particular, if $T\in M_{2m}(\Fq)$ is such that $l(c_T(p))>m$ for some $p\in \Irr\Fq[t]$, then $T$ does not admit an $m$-dimensional splitting subspace. 
\end{lemma}
\begin{proof}
  Let $\Pi_p:M_T\to M_{T_p}$ denote the projection map with respect to the primary decomposition \eqref{eq:primary}.
  Since $\Pi_p$ commutes with $T$, $\sum_{j\geq 0} T^j\Pi_p(W) = \Pi_p(\Fq^n)=M_{T_p}$.
  In other words, $\Pi_p(W)$ generates $M_{T_p}$.
  The $\Fq[t]$-module $M_{T_p}$ has rank $l(c_T(p))$, so any generating set must have at least $l(c_T(p))$ elements.
  Therefore, $\dim W\geq \dim \Pi_p(W)\geq l(c_T(p))$.
\end{proof}
Now $X^{\tau_i}_j(q)$ is given by \eqref{eq:taui}.
Therefore, in order to establish \eqref{eq:sigma-tau} for $\tau=\tau_i$, $i=1,\dotsc,m$, it suffices to prove the following result.
\begin{lemma}
  \label{lemma:vanishing-q_bins}
  For each positive integer $m$, $1\leq i\leq m$, and $0\leq k\leq m-i$,
  \begin{displaymath}
    \sum_{j=0}^{2m} (-1)^j \qbin{m+i}{j-k} q^{\binom{m-j+1}2}=0.
  \end{displaymath}
\end{lemma}
\begin{proof}
  In the $q$-binomial theorem
  \begin{displaymath}
    \sum_{j=0}^n \qbin nj q^{\binom j2}x^j  = \prod_{j=0}^{n-1}(1+q^jx),
  \end{displaymath}
  set $n=m+i$, $x=-q^{k-m}$, and change the index of summation from $j$ to $j+k$ to get
  \begin{equation}
    \label{eq:substituted}
    (-1)^k\sum_{j=k}^{m+i+k}(-1)^j\qbin{m+i}{j-k}q^{\binom{j-k}2+(k-m)(j-k)}=0.
  \end{equation}
  Observe that
  \begin{gather*}
    \binom{m-j+1}2 = [m(m+1) + j(j-1)-2mj]/2,\text{ whereas }\\
    \binom{j-k}2+(k-m)(j-k) = [k(k+1-2(k-m))+j(j-1)-2mj]/2.
  \end{gather*}
  These two expressions differ by a quantity independent of $j$.
  Therefore replacing $q^{\binom{j-k}2+(k-m)(j-k)}$ by $q^{\binom{m-j+1}2}$ in \eqref{eq:substituted} amounts to multiplication by a non-zero factor that is independent of $j$.
  Thus we have
  \begin{displaymath}
    \sum_{j=k}^{m+i+k}(-1)^j\qbin{m+i}{j-k}q^{\binom{m-j+1}2}=0.
  \end{displaymath}
  The sum remains unchanged when its range is extended to $0\leq j\leq 2m$, proving the identity in the lemma.
\end{proof}
\subsection*{Proof of Claim~\ref{claim:2}}
The non-singularity of $X=(X^{\tau_i}_j(q))_{0\leq i,j\leq m}$ is proved using inequalities satisfied by the degrees of its entries.
\begin{lemma}
  \label{lem:maxperm}
Let $(a_{ij})_{n\times n}$ be a real matrix such that whenever $i<k$ and $j<k$,
  $$
a_{ i k}-a_{ij}< a_{kk}-a_{kj}.
$$
Then the sum $S(\sigma)=\sum_{1\leq i\leq n}a_{i\sigma(i)}$ attains its maximum value precisely when $\sigma$ is the identity permutation.
\end{lemma}
\begin{proof}
  Let $\sigma $ be a permutation for which the sum $S(\sigma)$ is maximised. We claim that $\sigma(n)=n$. Suppose, to the contrary, that $\sigma(n)=s\neq n.$ Let $r=\sigma^{-1}(n)$. Now
  \begin{align*}
    \sum_{1\leq i\leq n}a_{i\sigma(i)}&=\sum_{i\notin \{r,n\}}a_{i\sigma(i)}+a_{rn}+a_{ns}\\
    &<\sum_{i\notin \{r,n\}}a_{i\sigma(i)}+a_{rs}+a_{nn}
  \end{align*}
  by the hypothesis since $r<n$ and $s<n$. If $\pi$ denotes the permutation which agrees with $\sigma$ whenever $i\notin \{r,n\}$ with $\pi(r)=s$ and $\pi(n)=n$, then it is clear that $S(\sigma)<S(\pi)$, contradicting the maximality of $S(\sigma)$. This proves the claim that $\sigma(n)=n$. Therefore 
  $$
S(\sigma)=a_{nn}+\max_{\pi \in S_{n-1}}S(\pi).
$$
Similar reasoning applied to the leading principal $(n-1) \times (n-1)$ submatrix of $A$ shows that $\sigma(n-1)=n-1$. Continuing this line of reasoning it can be seen that $\sigma(i)=i$ for each $i\leq n$, completing the proof.
\end{proof}
\begin{proposition}
  \label{prop:cofactor} 
  The matrix $X=(X_{j}^{\tau_i})_{0\leq i,j\leq m}$ is non-singular.
  \end{proposition}
  \begin{proof}
    Since $\tau_0$ is the type of a simple matrix, the first row of $X$ is the unit vector $(1,0,\dotsc,0)$.
    Therefore it suffices to show that the minor $X'=(X^{\tau_i}_j)_{1\leq i,j\leq m}$ is non-singular.
    Let $a_{ij}=\deg X^{\tau_i}_j(q)$.
    Since $\deg \qbin nk = (n-k)k$, by \eqref{eq:taui} we have, for $1\leq i,j\leq m$,
  \begin{displaymath}
   a_{ij}=\max\{j(m+i-j),(j-m+i)(2m-j)\}=j(m+i-j),
 \end{displaymath}
 since $j(m+i-j)-(j-m+i)(2m-j)=2(m-i)(m-j)\geq 0$.
 If $i<k$ and $j<k$,
  \begin{align*}
    a_{ik}-a_{ij}&=k(m+i-k)-j(m+i-j)\\
                 &=(k-j)(m+i-k-j)\\
                 &<(k-j)(m-j)\\
    &=a_{kk}-a_{kj}.
  \end{align*}
  Lemma \ref{lem:maxperm} implies that $\det X'$ has degree $\sum_{i=1}^ma_{ii}>0$ and is thus non-singular.
\end{proof}
\section{Chord Diagrams}
\label{sec:chord-diags}
A \emph{chord diagram} on $n=2m$ nodes refers to one of many visual representations of a fixed-point-free involution on $[2m]$ (see, e.g., \cite[Fig.~2]{MR1336847}).
We arrange $2m$ nodes along the $X$-axis.
A circular arc lying above the $X$-axis is used to connect each node to its image under the involution.
For example, the involution $(1,4)(2,6)(3,5)(7,8)$ is represented by the diagram
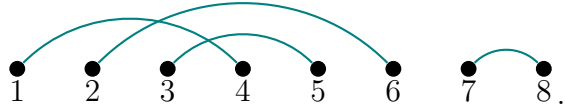
\begin{center}
  \begin{tikzpicture}
    [every node/.style={circle,fill=black,inner sep=0pt, minimum size=6pt}]
    \node[label=below:$1$] (1) at (1,0) {};
    \node[label=below:$2$] (2) at (2,0) {};
    \node[label=below:$3$] (3) at (3,0) {};
    \node[label=below:$4$] (4) at (4,0) {};
    \node[label=below:$5$] (5) at (5,0) {};
    \node[label=below:$6$] (6) at (6,0) {};
    \node[label=below:$7$] (7) at (7,0) {};
    \node[label=below:$8$] (8) at (8,0) {};
    \draw[thick,color=teal]
    (1) [out=45, in=135] to  (4);
    \draw[thick,color=teal]
    (2) [out=45, in=135] to  (6);
    \draw[thick,color=teal]
    (3) [out=45, in=135] to  (5);
    \draw[thick,color=teal]
    (7) [out=45, in=135] to  (8);
  \end{tikzpicture}.
\end{center}
The left end of each arc will be called an \emph{opening node}, and the right end a \emph{closing node}.
In the running example, the opening nodes are $1,2,3,7$ and the closing nodes are $4,5,6,8$.
A crossing of the chord diagram is a pair of arcs $(i,j),(k,l)$ such that $i<k<j<l$.
The chord diagram above has two crossings, namely $(1,4),(2,6)$ and $(1,4),(3,5)$.
Given a fixed-point-free involution $\sigma$, let $v(\sigma)$ denote the number of crossings of its chord diagram.
Touchard~\cite{MR46325} studied the polynomials
\begin{displaymath}
  T_m(q) = \sum_\sigma q^{v(\sigma)},
\end{displaymath}
where the sum runs over all fixed-point-free involutions of $[2m]$.

We now describe the contribution to $T_m(q)$ of chord diagrams with a specified set of opening nodes.
\begin{lemma}
  \label{lemma:openings}
  Given $1\leq c_1<\dotsb<c_m\leq 2m$ designated as opening nodes for a chord diagram, and the remaining elements of $[2m]$ designated as closing nodes of a chord diagram, $c_i$ lies to the left of the $j$th closing node if and only if
  \begin{displaymath}
    c_i \leq i+j-1.
  \end{displaymath}
  Consequently, the number of opening nodes that lie to the left of the $j$th closing node is given by
  \begin{equation}
    \label{eq:rj}
    r_j:= \#\{i\in [m]\mid c_i\leq j+i-1\}.
  \end{equation}
\end{lemma}
\begin{proof}
  The node $c_i$ lies to the left of the $j$th closing node of $\sigma$ if and only if there are at most $j-1$ closing nodes to the left of $c_i$.
  In other words, the total number of nodes (opening or closing) up to and including $c_i$ is at most $i+j-1$, meaning that $c_i\leq i+j-1$.
\end{proof}
For every non-negative integer $n$, let $[n]_q$ denote the $q$-integer $1+q+\dotsb+q^{n-1}$.
\begin{lemma}
  \label{lemma:Tq-refinement}
  For every non-negative integer $m$, and $1\leq c_1<\dotsb <c_m\leq 2m$,
  \begin{displaymath}
    \sum_{\sigma \text{ has opening nodes } c_1,\dotsc,c_m} q^{v(\sigma)}= \prod_{j=1}^m [r_j-(j-1)]_q,
  \end{displaymath}
  where $r_j$ is given by (\ref{eq:rj}).
 \end{lemma}
\begin{proof}
Suppose we wish to construct a chord diagram on $2m$ nodes with opening nodes $c_1<\dotsb<c_m$.
The remaining nodes $d_1<\dotsb<d_m$ are closing nodes.
By Lemma~\ref{lemma:openings}, for each $j\in [m]$, the number of opening nodes to the left of $d_j$ is $r_j$.
Thus there are $r_1$ choices of opening node for the arc ending at $d_1$.
These choices, taken from right to left, will result in $0,1,\dotsc,r_1-1$ crossings with arcs with closing nodes to the right of $d_1$.
Having chosen the node that is joined to $d_1$, the number of opening nodes that are available to $d_2$ is $r_2-1$.
Once again, these choices, taken from right to left, will result in $0,1,\dotsc,r_2-2$ crossings with arcs with closing nodes to the right of $d_2$.
Continuing in this manner, we see that the contribution of arcs with opening nodes $c_1,\dotsc,c_m$ to $T_m(q)$ is $\prod_{j=1}^m [r_j-(j-1)]_q$.
\end{proof}
Summing over all possible sets of opening nodes gives the following result.
\begin{theorem}
  For every non-negative integer $m$,
   \begin{displaymath}
    T_m(q) = \sum_{1\leq c_1<\dotsb <c_m\leq 2m}\: \prod_{j=1}^m [r_j-(j-1)]_q,
  \end{displaymath}
  where $r_j$ is given by (\ref{eq:rj}).
\end{theorem}
\section{The Enumeration of Splitting Subspaces}
\label{sec:enum-splitt-subsp}
The relationship between the enumeration of splitting subspaces and the polynomials $T_m(q)$ was discovered in \cite[Section~4.6]{pr}.
It is a special case of one of the main results \cite[Theorem~4.8]{pr} of that paper. 
The proof in this special case, being relatively simple, is provided here.
\begin{theorem}
  \label{theorem:split-splitting}
  Let $T\in M_{2m}(\Fq)$ be a diagonal matrix with distinct diagonal entries.
  The number of $T$-splitting subspaces in $\Fq^{2m}$ is
  \begin{displaymath}
    \sigma^T = (q-1)^m q^{\binom m2} T_m(q).
  \end{displaymath}
\end{theorem}
Comparison of Theorem~\ref{theorem:split-splitting} with the main theorem gives a new proof of the Touchard-Riordan formula \eqref{eq:touchard-riordan}.
\begin{proof}
  Let $W\subset \Fq^{2m}$ be a $T$-splitting subspace of $\Fq^{2m}$.
  $W$ has a unique ordered basis in reduced row echelon form.
  This is a basis whose elements form the rows of an $m\times 2m$ matrix such that
  \begin{enumerate}
  \item There exist $1\leq c_1<\dotsb < c_m\leq 2m$ (called the pivots of $W$) such that the first non-zero entry of the $i$th row lies in the $c_i$th column, and equals $1$.
  \item \label{item:rowech} The only non-zero entry in the $c_i$th column lies in the $i$th row for $1\leq i\leq m$.
  \end{enumerate}
  For example, when $m=3$, a subspace with pivots $1,2,4$ is spanned by a matrix of the form
  \begin{equation}
    \label{eq:rowech}
    \begin{pmatrix}
      1 & 0 & * & 0 & * & *\\
      0 & 1 & * & 0 & * & *\\
      0 & 0 & 0 & 1 & * & *
    \end{pmatrix},
  \end{equation}
  where each $*$ represents an arbitrary element of $\Fq$.
  
  Suppose that $W$ has reduced row echelon form with pivots $c_1,\dotsc,c_m$.
  By a permutation of coordinates, the pivot columns can be moved to the left to rewrite $A$ in the block form $(I\mid X)$, where $I$ denotes the $m\times m$ identity matrix, and $X\in M_m(\Fq)$.
  The condition \eqref{item:rowech} in the definition of row echelon from imposes the vanishing of certain entries of $X$:
  \begin{displaymath}
    X_{ij} = 0 \text{ if } j< c_i-(i-1).
  \end{displaymath}
  For the matrix in \eqref{eq:rowech}, moving the pivot columns to the left results in the matrix
  \begin{displaymath}
    \left(
      \begin{array}{ccc|ccc}
        1 & 0 & 0 & * & * & *\\
        0 & 1 & 0 & * & * & *\\
        0 & 0 & 1 & 0 & * & *
      \end{array}
    \right).
  \end{displaymath}
  The above permutation of coordinates also permutes the diagonal entries of $T$, but it remains a diagonal matrix with distinct diagonal entries.
  Write this matrix in block diagonal form as
  $\begin{pmatrix}
    T' & 0\\ 0 & T''
  \end{pmatrix}$, where $T'$ and $T''$ are $m\times m$ diagonal matrices.
  
  Now $W$ is a $T$-splitting subspace if and only if the matrix
  \begin{displaymath}
    \begin{pmatrix}
      I & X\\
      T' & XT''
    \end{pmatrix}
  \end{displaymath}
  is non-singular.
  Applying the block row operation $R_2\to R_2-T'R_1$ gives
  \begin{displaymath}
    \begin{pmatrix}
      I & X\\
      0 & XT''- T'X
    \end{pmatrix}
  \end{displaymath}
  Thus $W$ is a splitting subspace if and only if $Y=XT''-T'X$ is non-singular.
  The entries of $Y$ in terms of the entries of $X$ are given by
  \begin{displaymath}
    y_{ij} = (t''_j-t'_i)x_{ij},
  \end{displaymath}
  where $t'_i$ (resp. $t''_i$) is the $i$th diagonal entry of $T'$ (resp. $T''$).
  Since $T'$ and $T''$ have no diagonal entries in common the map $X\mapsto Y$ is a bijection, and an entry of $X$ is non-zero if and only if the corresponding entry of $Y$ is non-zero.
  Thus we have the following result.
  \begin{lemma}
    \label{lemma:counting-matrices}
    The number of $T$-splitting subspaces with pivots\linebreak $c_1,\dotsc,c_m$ is the number of non-singular matrices $Y\in M_m(\Fq)$ such that $Y_{ij}=0$ if $j<c_i-(i-1)$.
  \end{lemma}
  It remains to enumerate such matrices.
  The number of potentially non-zero entries in the $j$th column of $Y$ is the number of $i$ such that $c_i\leq i+j-1$, which is precisely the number $r_j$ from Lemma~\ref{lemma:openings}.
  Since $Y$ is non-singular, its first column is non-zero.
  Thus there are $q^{r_1}-1=(q-1)[r_1]_q$ possibilities for the first column of $Y$.
  The second column is independent of the first, giving $q^{r_2}-q=(q-1)q[r_2-1]_q$ possibilities.
  Similarly, given the first $j-1$ columns of $Y$, the number of possibilities for the $j$th column is $q^{r_j}-q^{j-1}=(q-1)q^{j-1}[r_j-(j-1)]_q$.
  Thus the number of matrices $Y$ satisfying the conditions of Lemma~\ref{lemma:counting-matrices} is
  \begin{displaymath}
    (q-1)^mq^{\binom m2}\prod_{j=1}^m[r_j-(j-1)]_q.
  \end{displaymath}
  Adding up the contribution of all possible sets of pivots and using Lemma~\ref{lemma:Tq-refinement} gives Theorem~\ref{theorem:split-splitting}.
\end{proof}
\printbibliography
\end{document}